\documentclass[letterpaper, 10pt, conference]{ieeeconf}
\usepackage{amsmath,amssymb,url,times,subfigure,graphicx,theorem}
\usepackage{graphicx,subfigure}
\usepackage{color,comment}
\usepackage{algpseudocode}
\usepackage{epstopdf}
\usepackage{siunitx}
\usepackage[all,import]{xy}
\usepackage{scalerel,stackengine}
\usepackage{tikz}
\usetikzlibrary{shapes.geometric, arrows}
\usepackage{float}

\newcommand{\norm}[1]{\ensuremath{\left\| #1 \right\|}}

\newcommand{\bracket}[1]{\ensuremath{\left[ #1 \right]}}

\newcommand{\refeqn}[1]{(\ref{eqn:#1})}

\newcommand{\deriv}[2]{\ensuremath{\frac{\partial #1}{\partial #2}}}

\newcommand{\SO}{\ensuremath{\mathsf{SO(3)}}}
\newcommand{\T}{\ensuremath{\mathsf{T}}}

\newcommand{\so}{\ensuremath{\mathfrak{so}(3)}}
\newcommand{\SE}{\ensuremath{\mathsf{SE(3)}}}

\newcommand{\Ree}{\ensuremath{\mathbb{R}}}
\newcommand{\Sph}{\ensuremath{\mathsf{S}}}

\renewcommand{\d}{\ensuremath{\mathbf{d}}}

\newcommand{\tr}[1]{\mbox{tr}\ensuremath{\negthickspace\bracket{#1}}}
\newcommand{\diag}[1]{\mbox{diag}\ensuremath{\negthickspace\bracket{#1}}}
\newcommand{\Lya}{\text{Lyapunov}}

\title{\LARGE \bf
	Geometric Adaptive Control for a Quadrotor UAV\\ with Wind Disturbance Rejection}

\author{Mahdis Bisheban, and Taeyoung Lee%
	\thanks{Mahdis Bisheban, Ph.D Candidate, Mechanical and Aerospace Engineering, George Washington University, Washington DC 20052 {\tt mbshbn@gwu.edu}}%
	\thanks{Taeyoung Lee, Associate Professor, Mechanical and Aerospace Engineering, George Washington University, Washington DC 20052 {\tt tylee@gwu.edu}}%
	\thanks{This research been supported in part by NSF under the grants CMMI-1243000, CMMI-1335008, and CNS-1337722}
}

\newtheorem{prop}{Proposition}

\graphicspath{{./Figs/}}

\begin{document}
\allowdisplaybreaks

\maketitle \thispagestyle{empty} \pagestyle{empty}

\begin{abstract}
	This paper presents a geometric adaptive control scheme for a quadrotor unmanned aerial vehicle, where the effects of unknown, unstructured disturbances are mitigated by a multilayer neural network that is adjusted online. The stability of the proposed controller is analyzed with Lyapunov stability theory on the special Euclidean group, and it is shown that the tracking errors are uniformly ultimately bounded with an ultimate bound that can be abridged arbitrarily. A mathematical model of wind disturbance on the quadrotor dynamics is presented, and it is shown that the proposed adaptive controller is capable of rejecting the effects of wind disturbances successfully. These are illustrated by numerical examples. 
\end{abstract}

\section{Introduction}


Quadrotor unmanned aerial vehicles (UAVs) have been utilized for various applications, such as aerial manipulation, search and rescue, and have been led to numerous research and developments. To complete outdoor flight missions successfully, it is required that the quadrotors are able to fly under wind disturbances reliably. However, it is challenging to accurately control quadrotor UAVs in a windy condition~\cite{Hoffmann2007}.

Several studies has been conducted to understand wind effects on a quadrotor. More specifically, wind affects the thrust and torque magnitude, and the drag force acting on the quadrotor. Also it causes blade flapping. For example, expressions for the thrust and drag force are presented in~\cite{Gill2017} for forward flights, and it is shown that if relative wind is more than $4$ to $7\si{\meter \per \second}$, the hover model deteriorates. In~\cite{Craig_Paley_16}, the effects of horizontal wind on the blade flapping are studied. In~\cite{Tomic_Haddadin_15}, it is discussed that it is critical to identify the drag and contact forces separately to detect collision or to estimate wind. In~\cite{Bisheban_Lee_IJCAS_17,Bisheban_Lee_CCTA_17}, computational geometric approaches for system identification of the rigid body and quadrotor dynamics are proposed, to estimate the unknown model parameters.


To reject the undesired effects of wind disturbances, different approaches have been taken. In~\cite{Tran15}, computational intensive aerodynamic expressions have been solved off-line, and a look-up table to estimate wind forces and moments in real-time has been used for control in simulation. 
In~\cite{Bangura_Mahony_17}, controllers are introduced in terms of the power and current on the motor to reach the desired thrust based on the aerodynamic power and motor model. The method is used to reject axial wind effects experimentally, while measurement of current, voltage and rotor speed are required during flight. 

While these methods depend on an accurate, estimated model of wind effects, other  controllers have been proposed to compensate the unknown un-modeled dynamics. For example, an adaptive controller based on the neural network is presented in~\cite{Nicol_Ramirez-Serrano_08} for hovering flights, which is robust against sinusoidal disturbances in simulations. Later, a nonlinear PID controller on the special Euclidean group is proposed in~\cite{Goodarzi_Lee_13} to reject unknown, fixed uncertainties. In~\cite{Mellinger_Kumar_2014}, controllers and trajectories are designed to reach desired states. To compensate for un-modeled dynamics, controllers are tuned from successive experimental trials data. Considering that models which are trained from off-line data may not perform well with new wind conditions, a controller based on neural network which learns on-line is proposed in~\cite{Dierks_Jagannathan_10}. The drag force and blade flapping effects as well as an unknown disturbance term for un-modeled dynamics are considered in the model, while attitude is represented with Euler angles.

This paper aims to present a geometric adaptive nonlinear controller based on the neural network to compensate for unknown un-modeled dynamics of a quadrotor UAV. This paper extends the controller proposed in~\cite{LeeLeoPICDC10}, with two three-layer neural networks for the coupled position and attitude dynamics. The controller is able to compensate for the unstructured disturbances in real-time, as the parameters of the neural network are adjusted online according to adaptive control laws. The quadrotor dynamics are studied directly on the special Euclidean group, namely $\SE$ to avoid singularities and complexities associated with Euler angles, or the ambiguities of quaternions. This is particularly useful for  the quadrotor to perform nontrivial aggressive maneuvers in uncertain environments. Through mathematical analysis of the $\Lya$ function on $\SE$, it is shown that the tracking errors are uniformly ultimately bounded with an ultimate bound that can be reduced arbitrarily up to any desired accuracy. We further present a comprehensive aerodynamic model of wind fields on a quadrotor, including the variations in thrust and torque magnitude, blade flapping effects and the drag force. With numerical results, it is shown that the controller can compensate for the un-modeled dynamics caused by wind disturbances, so to follow the desired trajectories successfully without any prior knowledge of the aerodynamics. 

In short, the main contribution of this paper is constructing a geometric adaptive controller on $\SE$ with the neural network to compensate unstructured disturbances acting on the both of the position dynamics and the attitude dynamics. Next, the efficacy of the proposed controller is validated by a comprehensive aerodynamic model of wind gusts. 

\section{Quadrotor Dynamics under Wind Disturbance}\label{sec:QDM}

\subsection{Quadrotor Dynamics}
Consider a quadrotor aerial vehicle composed of four identical rotors and propellers located at the vertices of a square, which generate a thrust and torque normal to the plane of this square. We define  an inertial reference frame $\{\vec e_1,\vec e_2,\vec e_3\}$ and a body-fixed frame $\{\vec b_1,\vec b_2,\vec b_3\}$. The origin of the body-fixed frame is located at the center of mass of this vehicle, and its third axis is pointing downward. 

The configuration space of a quadrotor is the special Euclidean group $\SE$, which is the semi-direct product of $\SO=\{R\in\Ree^{3\times 3}\,|\, R^TR=I_{3\times 3},\mathrm{det}[R]=+1\}$ and $\Ree^3$. For given $(R,x)\in\SE$, $R\in\SO$ represents the linear transform of representation of a vector from the body-fixed frame to the inertial frame, and $x\in\Ree^3$ denotes the location of the center of mass in the inertial frame. The kinematics equations are given by
\begin{align}
\dot x  &= v,\label{eqn:EC1}\\
\dot R & = R\hat\Omega,\label{eqn:EC3}
\end{align}
where $v\in\Ree^3$ is the linear velocity in the inertial frame, and $\Omega\in\Ree^3$ corresponds to the angular velocity resolved in the body-fixed frame. The hat map $\wedge :\Ree^3\rightarrow\so$ is defined  such that $\hat x y = x\times y$ and $(\hat x)^T =-\hat x$ for any $x,y\in\Ree^3$. The inverse of the hat map is denoted by the vee map $\vee :\so\rightarrow\Ree^3$.

Let $r_j\in\Ree^3$ be the location of the $j$-th rotor in the body-fixed frame for $j\in\{1,\ldots, 4\}$, defined as follows
\begin{gather}
r_1=\begin{bmatrix} d_h\\0\\d_v \end{bmatrix}, \;  r_2=\begin{bmatrix} 0\\-d_h\\d_v \end{bmatrix}, \;  r_3=\begin{bmatrix} -d_h\\0\\d_v \end{bmatrix}, \; r_4=\begin{bmatrix} 0\\d_h\\d_v \end{bmatrix},
\end{gather}
where $d_h,d_v\in\Ree$ specify the horizontal and vertical distances from the origin of the body-fixed frame to the rotors. In the absence of the relative wind, the $j$-th rotor generate the thrust $-T_je_3$ when resolved in the body-fixed frame, where $T_j\in\Ree$ denotes the magnitude of thrust.

Let the mass and the inertia matrix of the quadrotor be $m\in\Ree$, and $J\in\Ree^3$, respectively. The equations of motion are given by~\refeqn{EC1}--\refeqn{EC3} along with	
\begin{gather}
m\dot{v}=U_{e},\label{eqn:EC2}\\
J\dot \Omega + \Omega\times J\Omega = M_e,\label{eqn:EC4}
\end{gather}
where $U_e,M_e\in\Ree^3$ are the resultant force and the resultant moment acting on the quadrotor. Here, it is considered that $U_e$ is resolved in the inertial frame, and $M_e$ is resolved in the body-fixed frame. These includes the effects of the gravity, the aerodynamic forces, and the perturbation of the thrust due to the wind disturbances, and are formulated as follows. 

\begin{table}\centering
	\caption{Notations}\label{tab:notations}
	\begin{tabular}{cc}\hline
		Notation& Refers to  
		\\\hline
		$\hat{}$  & hat map $\hat{.}:\Ree^3\rightarrow\so$\\
		$\vee$ & vee map $.^\vee:\so\rightarrow\Ree^3$ \\
		$\bar{}$ & estimated value \\
		$\tilde{}$ & estimation error value  \\
		$\norm{}$ & Frobenius norm of a matrix, and 2-norm of a vector \\ 
		$\lambda_m{()}$ & minimum eigen value of a matrix \\	
		$\lambda_M{()}$ & maxmum eigen value of a matrix \\		\hline
	\end{tabular}
\end{table}

\subsection{Effects of Wind Disturbance}\label{Sec:Wind_effects}

Let $v_w\in\Ree^3$ be the velocity of the wind in the ambient atmosphere, resolved in the inertial frame. The relative wind on the $j$-th rotor in the body-fixed frame, namely $v_{w_j}\in\Ree^3$ is given by
\begin{equation}
v_{w_j}=R^T (v_w- v)+ \hat\Omega r_j\equiv[u_{1_j},u_{2_j},u_{3_j}]^T,\label{eqn:vrel}
\end{equation}
where the last two terms are caused by the motion of the quadrotor, and $u_{1_j}, u_{2_j}, u_{3_j} \in \Ree$ denote the components of $v_{w_j}$. 

The external resultant force acting on the quadrotor is given by
\begin{gather}
U_e=mge_3+D+R \Sigma_{j=1}^4 T_j d_j,\label{eqn:Uewind}
\end{gather}
where $mge_3$ is the gravitational force, and $D\in\Ree^3$ corresponds to the drag which is assumed to act on the center of mass in the opposite direction to the relative velocity of the quadrotor with respect to the wind. For a  positive drag coefficient $C_d\in\Ree$, it is given by
\begin{gather}
D=-C_d ||v-v_{w}|| (v-v_{w}).\label{eqn:Drag}
\end{gather}
The magnitude of the thrust for the $j$-th rotor is
\begin{gather}
T_j=C_{T_j}\rho A_p (r_p\omega_j)^2,\label{eqn:Twind}
\end{gather}
where  $C_{T_j}\in\Ree$ is the thrust coefficient, and $\rho\in\Ree$ is the air density. The constants $r_p$, $A_p=(\pi r_p)^2$, and $\omega_j$ are the radius, the sweeping area, and rotating speed of the $j$-th rotor, respectively. In \refeqn{Uewind}, the unit vector $d_j \in \Sph^2$ is the direction of $j$-th rotor thrust resolved in the body-fixed frame.  

The effects of wind are composed of the induced velocity that determines the thrust coefficient, and the blade flapping effects that alters the direction of the thrust. First, The thrust coefficient $C_{T_j}\in \Ree$ and the inflow ratio $\lambda_j\in \Ree$, which is the induced air velocity divided by the tip speed, are given by
\begin{align}
C_{T_j}&=\frac{s C_{l\alpha}}{2}[\theta_0(\frac{1}{3}+\frac{\mu_{x_j}^2}{2})-\frac{1}{2}(\lambda_j+\mu_{z_j})],\label{eqn:CT}\\
\lambda_j&=\frac{C_{T_j}}{2\sqrt{\mu_{x_j}^2+(\lambda_j+\mu_{z_j})^2}},
\end{align}
where $s=\frac{N_b c}{\pi r_p}\in\Ree$ is the solidity ratio which is the blade area approximated by $N_b c r_p$ divided by the blade sweeping area $\pi r_p^2$. Here, $c$ is the blade chord, and $N_b$ is the number of blades. The constants $C_{l\alpha},\theta_0\in\Ree$ are the blade lift curve slope and blade pitch angle respectively~\cite{Padfield2007}. The advance ratio parallel and perpendicular to the rotor plane are given by
\begin{gather}
\mu_{x_j}=\frac{\sqrt{u_{1_j}^2+u_{2_j}^2}}{\omega_j r_p},\\
\mu_{z_j}=\frac{u_{3_j}}{\omega_j r_p}.
\end{gather}
The expressions for the thrust coefficient $C_{T_j}$ and the inflow ratio $\lambda_j$ are implicit, and can be computed using Newton's iterative method~\cite{Padfield2007}.

Next, as the result of the blade flapping, the direction of the $j$-th rotor thrust changes from $-e_3$ to
\begin{gather}
d_{j}=\begin{bmatrix}\frac{-\sin{\alpha_j}}{\sqrt{u^2_{1_j}+u^2_{2_j}}} u_{1_j}\\ \frac{-\sin{\alpha_j}}{\sqrt{u^2_{1_j}+u^2_{2_j}}} u_{2_j}\\-\cos{\alpha_j} \end{bmatrix},\label{eqn:d_T}
\end{gather}
resolved in the body-fixed frame, where $\alpha_j$ is the blade flapping angle of the $j$-th rotor, resulting in the force component in the $b_1-b_2$ plane. This angle is approximated with
\begin{gather}
\alpha_j=C_{\alpha} \sqrt{u^2_{1_j}+u^2_{2_j}},\label{eqn:alpha}
\end{gather}
where $C_{\alpha}\in\Ree$ is the fixed flapping angle coefficient~\cite{Hoffmann2011b,Sydney2013}.

The external resultant moment consists of the moments due to the rotor thrusts, the blade flapping, and the reaction torques, i.e., 
\begin{align}
M_e=&\Sigma_{j=1}^4 r_j \times T_j d_j+(-1)^{j+1}Q_jd_j\nonumber\\
&+\frac{N_b}{2}K_\beta \alpha_j (d_j\cdot e_1+d_j\cdot e_2),\label{eqn:Mewind}\\
Q_j=&C_{Q_j}\rho A_p r_p(r_p\omega_j)^2,\label{eqn:Qwind}
\end{align}
where $K_\beta\in\Ree$ is the stiffness of the rotor blade~\cite{Hoffmann2011b},~\cite{Padfield2007} and $C_{Q_j}\in \Ree$ is the torque coefficient given by 
\begin{gather}
C_{Q_j}=C_{T_j}(\lambda_j+\mu_{z_j})+\frac{C_{D_0}s}{8}(1+3\mu_{x_j}^2), \label{eqn:CQ}
\end{gather}
where $C_{D_0}\in\Ree$ is the blade drag coefficient~\cite{Padfield2007}.

\section{Geometric Adaptive Controller}

In this section, we present a geometric adaptive control system for a quadrotor to reject the effects of wind disturbances without the knowledge of the aerodynamic model presented in Section~\ref{Sec:Wind_effects}. 

\subsection{Simplified Dynamic Model}

We first formulate a simplified dynamic model for the control system development, where the effects of wind are considered as unstructured, unknown disturbances to the quadrotor. 

In other words, it is assumed that
\begin{gather}
T_j^\prime=C_T^\prime\omega_j^2,\quad
Q_j^\prime=C_Q^\prime\omega_j^2\equiv C_{TQ}T_j^\prime,\\
D=0,\quad d_j=-e_3,
\end{gather}
where $T_j^\prime,Q_j^\prime\in \Ree$ are the magnitude of the thrust and torque of the $j$-th rotor respectively. The constants $C_T^\prime,C_Q^\prime\in \Ree$ are the identical constant thrust and torque coefficients for all rotors, and $C_{TQ}\in \Ree$ describes the reactive torque based on the rotor thrust. Here, the superscript $\prime$ is used to distinguish the model parameters which is used for controller from the model parameters presented in Section~\ref{Sec:Wind_effects}.  The corresponding resultant force and moment are given by
\begin{gather}
U_e^\prime=mge_3-fRe_3,\label{eqn:E3}\\
M_e^\prime=-\Sigma_{j=1}^4 r_j \times T_j^\prime e_3-(-1)^{j+1}Q_j^\prime e_3,\label{eqn:E4}
\end{gather}
where $f=\Sigma_{j=1}^4T_j\in\Ree$ is the sum of four rotor thrusts. As we are aware that the model, which is used to design the controller in this section, is not accurate and differs from the physical quadrotor or the model presented in~\ref{Sec:Wind_effects}, we consider two unknown terms $\Delta_1,\Delta_2\in\Ree^3$ as the model errors. More precisely, the simplified dynamic model considered for the control system development is given by
\begin{gather}
m \dot{v}=U_e^\prime-\Delta_1,\label{eqn:E1}\\
J \dot{\Omega}+\Omega \times J \Omega=M_e^\prime-\Delta_2.\label{eqn:E2}
\end{gather}
with the kinematics equations~\refeqn{EC1}--\refeqn{EC3}.

\subsection{Tracking Problem Formulation}

Suppose that there is a desired position trajectory of the quadrotor, namely $x_d(t)\in\Ree^3$ given as a smooth function of time. We wish to develop a control system such that the controlled trajectory errors are uniformly ultimately bounded. 

The proposed control system extends the geometric tracking control scheme presented in~\cite{LeeLeoPICDC10}, by including a neural-network based adaptive control term. The key idea is that by utilizing the universal function approximation property of a multilayer neural network, we mitigate the effects of the unknown disturbances by adjusting the weights of neural network according to an adaptive control law. 

More explicitly, the proposed control input for the total thrust $f$ and torque $M_c$ are given by
\begin{align}
f=&-A^T Re_3,\label{eqn:f}\\
A=&\bar \Delta_1-k_x e_x-k_v e_v -mge_3+m \ddot x_d,\label{eqn:A}\\
M_c=&\bar \Delta_2-k_R e_R-k_\Omega e_\Omega+\Omega \times J \Omega\nonumber\\
&-J(\hat \Omega R^T R_c \Omega_c-R^T R_c \dot{\Omega}_c),\label{eqn:Mc}
\end{align}
where $k_x,k_v,k_R,k_\Omega$ are positive constants. $\bar \Delta_1, \bar \Delta_2$ are the adaptive control terms, and are defined in Section~\ref{Sec:ANN}. The matrix $R_c\in\SO$, and the vectors $\Omega_c,\dot{\Omega}_c\in\Ree^3$ are the computed rotation matrix, angular velocity, and its first derivative. 

The control moment $M_c$ is designed such that the actual rotation matrix $R$ follows the computed rotation matrix $R_c$. The three columns of $R_c$ are denoted by $R_c\equiv[b_{1c},b_{2c},b_{3c}]$, where $b_{1c},b_{2c},b_{3c}\in\Sph^2$. To design $R_c$, first, we set $b_{3c}=-\frac{A}{||A||}$ for $A$ defined in \refeqn{A}. This is to generate the total thrust to follow the given desired position trajectory. To ensure $R_c\in\SO$, $b_{1c}$ must be orthogonal to $b_{3c}$. So, there is one-dimensional degree of choice to select $b_{1c}$ which is actually the heading direction, more specifically the yaw angle.  We define $b_{1d}\in\Sph^2$ such that $b_{1d}\times b_{3c}\neq0$. Then, we restrict $b_{1c}$ to be the projection of $b_{1d}$ on the plane perpendicular to $b_{3c}$. More explicitly,
\begin{gather}
R_c=[b_{2c}\times b_{3c},-\frac{C}{||C||},-\frac{A}{||A||}],\label{eqn:Rc}\\
C=-b_{3c} \times b_{1d},\label{eqn:C}
\end{gather}
It is assumed that $\norm{A}\neq0$, $\norm{\bar\Delta_1}\leq\delta_1$, where $\delta_1$ is positive, and the command acceleration is uniformly bounded such that
\begin{gather}\norm{-mge_3+m \ddot x_d+\bar \Delta_1}\leq B_1,\label{eqn:acceleration_bound}
\end{gather} 
for a given positive constant $B_1$. Next, the computed angular velocity is defined as
\begin{gather}
\Omega_c=(R_c^T\dot{R}_c)^\vee
\end{gather}

The control inputs given in \refeqn{A} and \refeqn{Mc} are also dependent of the tracking errors $e_x,e_v,e_R,e_v\in\Ree^3,\Psi\in \Ree$ are defined as
\begin{gather}
e_x=x-x_d,\label{eqn:ex}\\
e_v=v-\dot{x}_d,\label{eqn:ev}\\
e_R=\frac{1}{2}(R_c^TR-R^TR_c)^\vee,\label{eqn:eR}\\
e_\Omega=\Omega-R^TR_c\Omega_c,\label{eqn:eomega}\\
\Psi(R,R_c)=\frac{1}{2}\tr{I_{3\times 3}-R_c^TR},\label{eqn:Psi}
\end{gather}
where $\Psi$ is positive-definite about $R=R_c$, that is a unique critical point in $\mathcal{D}_0=\{R\in \SO|\Psi(R,R_c)<2\}$~\cite{LeeLeoPICDC10}. 

Having $f$ and $M_c\equiv[M_1,M_2,M_3]^T$ determined by controller, the four rotor thrusts can be computed using
\begin{align}
\begin{bmatrix} T_1^\prime\\T_2^\prime\\T_3^\prime\\T_4^\prime\end{bmatrix}=\begin{bmatrix}1&1&1&1\\ 0&-d_h&0&-d_h\\d_h&0&-d_h&0\\-C_{TQ}&C_{TQ}&-C_{TQ}&C_{TQ}  \end{bmatrix}^{-1} \begin{bmatrix} f\\M_1\\M_2\\M_3\end{bmatrix}.
\end{align}

\subsection{Adaptive Neural Network Structure}\label{Sec:ANN}

In this section, we present how the adaptive control terms are defined in~\refeqn{A} and \refeqn{Mc}, which are the main contribution of this paper.  

According to the universal approximation theorem~\cite{lewis1998neural}, a multilayer artificial neural network can approximate any continuous function up to an arbitrary accuracy. Specifically, there are ideal constant weight matrices $W$ and $V$, and the number of hidden layers such that approximation of the function for a given desired positive accuracy $\varepsilon_N$ is given by 
\begin{gather}
\Delta(x_{nn})=W^T \sigma(V^T x_{nn})+\varepsilon{(x_{nn})},
\end{gather}
where $x_{nn}$ is the input vector to the neural network, $\sigma$ is the activation function, and the error satisfies $\norm{\varepsilon{(x_{nn})}}\leq\varepsilon_N$.

Here, to compensate for the uncertainties $\Delta_1,\Delta_2$ presented in~\refeqn{E1}--\refeqn{E2}, we use two neural networks: one for the position dynamics and another for the attitude dynamics. Throughout the paper, the subscript $i=1$ is to refer to the position dynamics, and $i=2$ is to refer to the attitude dynamics. As we have no information about $\Delta_i$ in~\refeqn{E1}--\refeqn{E2}, we use the estimates of the ideal weights in the control system such that $\bar \Delta_i=\bar W_i^T \sigma(\bar{z}_i)$ where $\bar{z}_i=\bar V_i^T x_{nn_i}$. The number of neurons in the first, second or hidden, and the last layers are denoted by $N_{1_i},N_{2_i},N_{3_i}$ respectively. Thus, $W_i \in \Ree^{N_{2_i}+1, N_{3_i}}$, $V_i \in \Ree^{N_{1_i}+1, N_{2_i}}, x_{nn_i} \in \Ree^{N_{1_i}+1}$. 

We define $Z_i=\diag{W_i,V_i}\in \Ree^{N_{2_i}+N_{1_i}+2, N_{2_i}+N_{3_i}}$ and the errors in the neural network as follows
\begin{gather}
\tilde{W_i}=W_i-\bar W_i,\quad \tilde{V_i}=V_i-\bar V_i,\quad \tilde{Z_i}=Z_i-\bar Z_i.\label{eqn:errorNN}
\end{gather}
The weights are bounded using the gradient projection method~\cite{IoannouRobustadaptiveComtrol} such that 
\begin{gather}
 ||W_i|| \leq W_{M_i}, \quad  ||V_i|| \leq V_{M_i},\quad ||Z_i|| \leq Z_{M_i},\label{eqn:NNbound}
\end{gather}
where $W_{M_i},W_{M_i},W_{M_i} \in \Ree^{+}$. The output error of the neural network $\tilde{\Delta}_i=\Delta_i-\bar{\Delta}_i$ is written as
\begin{gather}
\tilde{\Delta}_i=\tilde W_i^T [\sigma(\bar{z}_i)-\sigma^\prime(\bar{z}_i)\bar{z}_i]+\bar W_i^T \sigma^\prime(\bar{z}_i)\tilde{z}_i-w_i,\label{eqn:OENN}\\
w_i=-\tilde W_i \sigma^\prime(\bar{z}_i) z_i-W_i^T\mathcal{O}_i-\varepsilon(x_{nn_i}),\label{eqn:wi}\\
\mathcal{O}_i=\sigma(z_i)-\sigma(\bar{z}_i)-\sigma^\prime(\bar{z}_i) \tilde{z}_i,\label{eqn:O}
\end{gather}
 where $z_i= V_i^T x_{nn_i},\,\tilde{z}_i= \tilde{V}_i^T x_{nn_i}$, and $\sigma: \Ree^{N_{2_i}} \mapsto \Ree^{N_{2_i}+1} $. $\sigma(z_i)=[1,\varsigma_1,\ldots,\varsigma_{N_{2_i}}]^T$  contains the sigmoid activation functions $\varsigma_k=\frac{1}{1+e^{-z_k}}$, where $k\in{1,\ldots,N_{2_i}}$ and $z_k$s are the elements of the vector $z_i$, and $\sigma^\prime(z_i)=[0,\deriv{\varsigma_1}{z_1},\ldots,\deriv{\varsigma_{N_{2_i}}}{z_{N_{2_i}}}]^T$, where $\deriv{\varsigma_k}{z_k}=\varsigma_k(1-\varsigma_k)$. One can show $w_i$ is bounded such that
\begin{gather}
\norm{w_i}\leq C_{1_i}+\norm{\tilde{Z}_i}(C_{2_i}+C_{3_i} \norm{x_{1_i}}+C_{4_i} \norm{x_{2_i}}),
\end{gather} 
where $C_{k_i}, k\in{1,\ldots,4}$ are positive constants~\cite{LeeKimJGCD01}. The input to the neural networks are defined as
\begin{gather}
x_{nn_i}=[1,x_{1_i},x_{2_i}],\\
x_{1_1}=x,\quad x_{2_1}=v,\label{eqn:x12_1}\\ 
x_{1_2}=E(R)^T,\quad x_{2_2}=\Omega,\label{eqn:x12_2}
\end{gather}
where $E(R)=[\theta, \phi, \psi]$ contains the Euler angles from the rotation matrix $R$. 

In the proposed adaptive control system, the neural network weights are updated according to
\begin{gather}
\dot{\bar{W_i}}=-\gamma_{w_i}[\sigma(z_i)a_i^T- \sigma^\prime(z_i)z_ia_i^T]-\kappa_i \gamma_{w_i} \bar W_i,\label{eqn:Wdot}\\
\dot{\bar{V_i}}=-\gamma_{v_i}x_{nn_i}[ \sigma^\prime(z_i)^T\bar W_ia_i]^T-\kappa_i \gamma_{v_i} \bar V_i,\label{eqn:Vdot}\\
a_1=e_v+c_1 e_x,\quad a_2=e_\Omega+c_2  e_R,\label{eqn:a1a2}
\end{gather}
where $\gamma_{w_i}$, $\gamma_{v_i}$ and $\kappa_i$ are the positive design parameters.

The stability properties of the proposed control system are summarized as follows. 

\begin{prop}{}\label{prop:2}
Consider the control force $f$ and moment $M_c$ defined at \refeqn{f}, \refeqn{Mc}. Suppose that the initial condition satisfies
	\begin{align}
	\Psi(R(0),R_d(0))\leq \psi_1 < 1,\label{eqn:Psi0}\\
	\norm{e_x(0)}<e_{x_{max}},
	\end{align}
	for fixed constants $\psi_1$ and $e_{x_{max}}$. For any positive constants $k_x,k_v$, we choose positive constants $c_1,c_2,k_R,k_\Omega,\gamma_{w_1},\gamma_{w_2},\gamma_{v_1},\gamma_{v_2},\kappa_1,\kappa_2$ such that they satisfy~\refeqn{c10},~\refeqn{c20}, and~\refeqn{N1}--\refeqn{N3} become positive-definite.
	
Then, all of the tracking errors of the quadrotor UAV, as well as the neural network weight errors are uniformly ultimately bounded with the ultimate bound that can be abridged arbitrary through~\refeqn{D}.
\end{prop}

\begin{proof} 
The proof is presented in Appendix. To see the definition of uniformly ultimately bounded errors refer to~\cite{khalil1996noninear}.
\end{proof}

\appendix
\subsection{Identities}
In this section, selected identities that are used throughout the proof are presented.

For any $\mathcal{A}\in \Ree^{3\times 3},x, y\in\Ree^{3}$, $c_1, c_2,c_3\in \Ree$, 
\begin{gather}
\tr{ y x^T}=x^Ty,\label{eqn:tryx}\\
||x+y||\leq||x||+||y||,\label{eqn:x_y}\\
-c_1x^2+c_2x=-\frac{c_1}{2}x^2-\frac{c_1}{2}[x-\frac{c_2}{c_1}]^2+\frac{c_2^2}{2c_1}\nonumber\\
\leq-\frac{c_1}{2}x^2+\frac{c_2^2}{2c_1},\label{eqn:ax2+bx}\\
-c_1x^2-c_2xy-c_3y^2\leq-c_1x^2+c_2xy-c_3y^2.\label{eqn:quad}
\end{gather}
A positive-definite function denoted by $\mathcal{V}_{0_i}$, and its derivative, which are used in the subsequent stability analysis, are defined as 
\begin{gather}
\mathcal{V}_{0_i}=\frac{1}{2\gamma_{w_i}}\tr{\tilde W_i^T \tilde W_i}+\frac{1}{2\gamma_{v_i}}\tr{\tilde V_i^T \tilde V_i},\label{eqn:Lyap0}\\
\dot{\mathcal{V}}_{0_i}=\frac{1}{\gamma_{w_i}}\tr{\tilde W_i^T \dot{\tilde{W}}_i}+\frac{1}{\gamma_{v_i}}\tr{\tilde V_i^T \dot{\tilde{V}}_i},\label{eqn:Lyap0_dot}
\end{gather} 
where $\tilde W_i,\,\tilde V_i$ are defined in~\refeqn{errorNN}. 

Next, we find the upper bound of the following expression, defined as $\mathcal{B}_i$
\begin{align}
\mathcal{B}_i&\equiv -a_i^T(\tilde\Delta_i)+\dot{\mathcal{V}}_{0_i}.\label{eqn:B}
\end{align}
The error dynamics of the neural network weights from~\refeqn{errorNN} are give by
\begin{gather}
\dot{\tilde{W}}_i=-\dot{\bar{W}}_i,\quad
\dot{\tilde{V}}_i=-\dot{\bar{V}}_i.\label{eqn:tildeWV}
\end{gather}
We substitute~\refeqn{Wdot}--\refeqn{Vdot} in~\refeqn{tildeWV}, and the result in~\refeqn{Lyap0_dot}, then using~\refeqn{OENN}, $\mathcal{B}_i$ is rewritten as
\begin{align}
\mathcal{B}_i&=a_i^T\{-\tilde W_i^T [\sigma(z_i)-\sigma^\prime(z_i)z_i]-\bar W_i^T \sigma^\prime(z_i)\tilde{z}_i+w_i\}\nonumber \\
&\quad+\tr{\tilde W_i^T[\sigma(z_i)a_i^T-\sigma^\prime(z_i)z_ia_i^T+\kappa_i \bar W_i]} \nonumber \\
&\quad+\tr{\tilde V_i^T\{x_{nn_i}[\sigma^\prime(z_i)^T \bar W_i a_i]^T+\kappa_i \bar V_i \}}.
\end{align}
Applying~\refeqn{tryx}, it reduces to
\begin{gather}
\mathcal{B}_i=\kappa_i \tr{\tilde Z_i^T \bar Z_i}+a_i^T(w_i).\label{eqn:Bi_init}
\end{gather}
The following is hold for $\tr{\tilde Z_i^T \bar Z_i}$
\begin{gather}
\tr{\tilde Z_i^T \bar Z_i}=\tr{\tilde Z_i^T Z_i}-\tr{\tilde Z_i^T \tilde Z_i}\leq||\tilde Z_i|| Z_{M_i}-||\tilde Z_i||^2.
\end{gather}
\refeqn{ax2+bx} implies
\begin{gather}
- ||\tilde Z_i||^2+Z_{M_i} ||\tilde Z_i||\leq-\frac{1}{2}||\tilde Z_i||^2+\frac{ Z_{M_i}^2}{2}.  \label{eqn:traceZtilde}
\end{gather}
Since $\norm{\sigma}\leq 1,\, \norm{\sigma^\prime}\leq 0.25$, it can be shown that the upper bound for~\refeqn{O} is
\begin{gather}
\norm{\mathcal{O}_i} \leq 2+0.25\norm{\tilde V_i} \norm{x_{nn_i}}.
\end{gather}
From~\refeqn{NNbound}, the upper bound of~\refeqn{wi} is
\begin{align}
\norm{w_i}\leq&
0.25 V_{M_i} \norm{\tilde W_i} \norm{x_{nn_i}}+W_{M_i} \norm{\mathcal{O}_i}+\epsilon_i.\label{eqn:normw}
\end{align}
Since $\norm{x_{nn_i}}\leq 1+\norm{x_{1_i}}+\norm{x_{2_i}}$, $\norm{\tilde Z_i}\geq\norm{\tilde W_i},\,\norm{\tilde Z_i}\geq\norm{\tilde V_i}$,~\refeqn{NNbound}, and substituting $\norm{\mathcal{O}_i}$, we get 
\begin{gather}
\norm{w_i}\leq C_{1_i}+C_{2_i}||\tilde Z_i|| (1+\norm{x_{1_i}}+\norm{x_{2_i}}),
\end{gather}
where $C_{2_i}\geq 0.25 (V_{M_i}+W_{M_i}),\,C_{1_i}\geq 2W_{M_i}+\epsilon_i$. 

Substituting~\refeqn{traceZtilde} and~\refeqn{normw} in~\refeqn{Bi_init} results in
\begin{align}
\mathcal{B}_i\leq&-\frac{\kappa_i}{2}||\tilde Z_i||^2+\frac{\kappa_i Z_{M_i}^2}{2}\nonumber\\&+\norm{a_i}\{C_{1_i}+C_{2_i}||\tilde Z_i|| (1+\norm{x_{1_i}}+\norm{x_{2_i}})\}.\label{eqn:last_term}
\end{align}

\subsection{Position Error Dynamics} \label{Apen:position}
In this section, we analyze the error dynamics for the position tracking command, which will be integrated with the attitude error dynamics in Appendix~\ref{Apen:quadrotor} for the stability proof of the complete dynamics. 

Taking derivative of~\refeqn{ex}--\refeqn{ev} and substituting~\refeqn{E3} and~\refeqn{E1}, the error dynamics are defined as
\begin{gather}
\dot{e}_x=e_v,\\
m\dot{e_v}=mge_3-\Delta_1-fRe_3-m\ddot{x}_d.\label{eqn:mevdot1}
\end{gather}
Define $\mathcal{X}\in\Ree^3$ as
\begin{gather}
\mathcal{X}\equiv\frac{f}{e_3^TR_c^TRe_3}[(e_3^TR_c^TRe_3)Re_3-R_ce_3],
\end{gather}
where $e_3^TR_c^TRe_3>0$~\cite{LeeLeoPICDC10}. Equation \refeqn{mevdot1} is rewritten as
\begin{gather}
m\dot e_v=mge_3-\Delta_1-m\ddot x_d-\frac{f}{e_3^TR_c^TRe_3}R_ce_3-\mathcal{X}.\label{eqn:mevdot2}
\end{gather}
Since $b_{3c}=R_ce_3=\frac{-A}{\norm{A}}$, $f=-A^TRe_3$, we can conclude that $f=(\norm{A}R_ce_3)^TRe_3$, therefore
\begin{gather}
-\frac{f}{e_3^TR_c^TRe_3}R_ce_3=A.\label{eqn:newA}
\end{gather}
Then substituting~\refeqn{newA},~\refeqn{A} in~\refeqn{mevdot2}, the velocity error dynamics is given by
\begin{gather}
m\dot e_v=-k_x e_x-k_ve_v-\tilde \Delta_1-\mathcal{X}.\label{eqn:mevdot}
\end{gather}

Next, we find the upper bound of $\mathcal{X}$. From~\refeqn{newA}, $\norm{A}=\norm{-\frac{f}{e_3^TR_c^TRe_3}R_ce_3}$. $R_ce_3$ is the unit vector, so $\norm{A}=\norm{-\frac{f}{e_3^TR_c^TRe_3}}$. Consequently, the norm of $\mathcal{X}$ can be written as
\begin{gather}
\norm{\mathcal{X}}=\norm{A}\norm{[(e_3^TR_c^TRe_3)Re_3-R_ce_3}.
\end{gather}
Also, it is shown that $\norm{[(e_3^TR_c^TRe_3)Re_3-R_ce_3}\leq\norm{e_R}\leq\beta<1$, where $\beta=\sqrt{\psi_1(2-\psi_1)}$~\cite{LeeLeoPICDC10}. Substituting~\refeqn{A} and~\refeqn{acceleration_bound}, the upper bound of $\norm{\mathcal{X}}$ is given by
\begin{gather}
\norm{\mathcal{X}}\leq(k_x\norm{e_x}+k_v\norm{e_v}+B_1)\norm{e_R}.\label{eqn:normX}
\end{gather}

For a non-negative constant $c_1$, the Lyapunov function for the position dynamics is chosen as
\begin{gather}
\mathcal{V}_1=\frac{1}{2}k_x e_x^T e_x+\frac{1}{2}m e_v^T e_v+m c_1 e_x^T e_v+\mathcal{V}_{0_1},\label{eqn:V1}
\end{gather}
where $\mathcal{V}_{0_1}$ is given by~\refeqn{Lyap0}. It is straightforward to show
\begin{gather}
\lambda_m(\mathcal{M}_{11}) ||\mathcal{Z}_{11}||^2+\mathcal{V}_{0_1}\leq\mathcal{V}_1 \leq \lambda_M(\mathcal{M}_{12}) ||\mathcal{Z}_{11}||^2+\mathcal{V}_{0_1}, \label{eqn:V1bound}
\end{gather}
where
\begin{gather}
\mathcal{M}_{11}=\frac{1}{2}\begin{bmatrix} k_x & -mc_1
\\ -mc_1 & m \end{bmatrix},\quad\mathcal{M}_{12}=\frac{1}{2}\begin{bmatrix} k_x & mc_1
\\ mc_1 & m \end{bmatrix},\label{eqn:M1}\\
\mathcal{Z}_{11}=[||e_x||,||e_v||]^T.\label{eqn:Z11}
\end{gather} 
If $c_1$ is sufficiently small such that
\begin{gather}
c_1<\sqrt{\frac{k_x}{m}},\label{eqn:c10}
\end{gather}
then $\mathcal{M}_{11},\mathcal{M}_{12}$ are positive-definite. 

Taking derivate of the $\Lya$ function results in
\begin{align}
\dot{\mathcal{V}}_1=&k_x e_v^Te_x+(e_v+c_1e_x)^Tm\dot e_v+mc_1e_v^Te_v+\dot{\mathcal{V}}_{0_1},\label{eqn:vdot_mid1}
\end{align}
where $\dot{\mathcal{V}}_{0_1}$ is given by~\refeqn{Lyap0_dot}. Substituting~\refeqn{mevdot} to~\refeqn{vdot_mid1}, and rearranging it result in
\begin{align}
\dot{\mathcal{V}}_1=&(mc_1-k_v)e_v^Te_v-c_1 k_x e_x^T e_x-c_1k_v e_x^T e_v\nonumber \\
&-(e_v+c_1e_x)^T\mathcal{X}-(e_v+c_1e_x)^T\tilde \Delta_1+\dot{\mathcal{V}}_{0_1}.
\end{align}

Since~\refeqn{a1a2}, the last term of the above expression is the same as~\refeqn{B}. Substituting its equivalent expression given by~\refeqn{last_term}, and substituting~\refeqn{normX} result in
\begin{align}
\dot{\mathcal{V}}_1\leq&-(k_v(1-\beta)-mc_1)e_v^Te_v-c_1k_x(1-\beta)e_x^Te_x\nonumber \\&+c_1k_v(1+\beta)\norm{e_x}\norm{e_v}-\frac{\kappa_1}{2}||\tilde Z_1||^2+\frac{\kappa_1 Z_{M_1}^2}{2}\nonumber\\
&+\norm{a_1}\{C_{1_1}+C_{2_1}||\tilde Z_1|| (1+\norm{x_{1_1}}+\norm{x_{2_1}})\}\nonumber \\&+\norm{e_R}\{B_1(c_1\norm{e_x}+\norm{e_v})+k_xe_{x_{max}}\norm{e_v}\},\label{eqn:Vdot1_mid0}
\end{align}
where $\norm{e_x}\leq e_{x_{\max}}$, for a fixed positive constant $e_{x_{\max}}$, is used for simplifying multiplication of the three vectors, $\norm{e_R}\norm{e_x}\norm{e_v}$.

It is assumed that the desired trajectory is bounded such that $\norm{x_d}\leq x_{d_{max}},\,\norm{\dot x_d}\leq v_{d_{max}}$, where $x_{d_{max}},\, v_{d_{max}}>0$. These as well as~\refeqn{ex}--\refeqn{ev}, and~\refeqn{x12_1} imply $\norm{x_1}\leq\norm{e_x}+x_{d_{max}}$, $\norm{x_2}\leq\norm{e_v}+v_{d_{max}}$. Substituting these in~\refeqn{Vdot1_mid0}, expanding $a_1$, and using~\refeqn{x_y} result in
\begin{align}
\dot{\mathcal{V}}_1\leq&-(k_v(1-\beta)-mc_1)e_v^Te_v-c_1k_x(1-\beta)e_x^Te_x\nonumber \\&+c_1k_v(1+\beta)\norm{e_x}\norm{e_v}-\frac{\kappa_1}{2}||\tilde Z_1||^2+\frac{\kappa_1 Z_{M_1}^2}{2}\nonumber\\
&+C_{1_1}\norm{e_v} +C_{3_1} \norm{e_v}^2+C_{4_1}\norm{e_v}\norm{\tilde Z_1}\nonumber\\
&+c_1(C_{1_1}\norm{e_x}+C_{3_1}\norm{e_x}^2+C_{4_1}\norm{e_x}\norm{\tilde Z_1})\nonumber\\&+(1+c_1)C_{3_1}\norm{e_x}\norm{e_v}\nonumber \\&+\norm{e_R}\{B_1(c_1\norm{e_x}+\norm{e_v})+k_xe_{x_{max}}\norm{e_v}\},\label{eqn:Vdot1_mid1}
\end{align}
where $C_{1_1}\geq 2W_{M_1}+\epsilon,\,C_{2_1}\geq 0.25 (V_{M_1}+W_{M_1}),\,C_{3_1}\geq C_{2_1} Z_{M_1},\,C_{4_1}\geq C_{2_1}(1+x_{d_{max}}+v_{d_{max}})$. 

Using~\refeqn{ax2+bx}, and defining $k_{v_{\beta}}\equiv k_v(1-\beta)-mc_1-C_{3_1}$ and $k_{x_\beta}\equiv k_x(1-\beta)-C_{3_1}$, the following expressions are rearranged as
\begin{align}
-&k_{x_\beta}e_x^Te_x+C_{1_1}e_x\leq-\frac{k_{x_\beta}}{2}e_x^Te_x+\frac{C_{1_1}^2}{2k_{x_\beta}},\label{eqn:one1}\\
-&k_{v_{\beta}}e_v^Te_v+C_{1_1}e_v\leq-\frac{k_{v_{\beta}}}{2}e_v^Te_v+\frac{C_{1_1}^2}{2k_{v_{\beta}}},\label{eqn:one2}
\end{align}
Substituting~\refeqn{one1}--\refeqn{one2} into~\refeqn{Vdot1_mid1} results in 
\begin{align}
\dot{\mathcal{V}}_1\leq&-\frac{c_1k_{x_\beta}}{2}e_x^Te_x-\frac{k_{v_\beta}}{2}e_v^Te_v-\frac{\kappa_1}{2}||\tilde Z_1||^2+k_{xv}\norm{e_x}\norm{e_v}\nonumber \\
&+C_{4_1}||e_v||||\tilde Z_1||+c_1C_{4_1}||e_x||||\tilde Z_1||+C_{5_1}\nonumber\\
&+\norm{e_R}\{c_1B_1\norm{e_x}+(B_1+k_xe_{x_{max}})\norm{e_v}\},\label{eqn:V1dot}
\end{align}
where $k_{xv}=c_1[(1+\beta)k_v+C_{3_1}]+C_{3_1}$, 
 $C_{5_1}=\frac{c_1C_{1_1}^2}{2k_{x_\beta}}+\frac{C_{1_1}^2}{2k_{v_\beta}}+\frac{\kappa_1 Z_{M_1}^2}{2}$.

\subsection{Attitude Error Dynamics} \label{Apen:attitude}
Here, we analyze the error dynamics for the attitude tracking command. The attitude error dynamics are defined as
\begin{gather}
	\dot{e}_R=\frac{1}{2}(\tr{R^T R_c}I_{3\times 3}-R^TR_c)e_\Omega \equiv C(R_c^TR)e_\Omega,\label{eqn:eRdot}\\
	J\dot{e_\Omega}=-k_Re_R-k_\Omega e_\Omega-\tilde\Delta_2,\label{eqn:eomegadot}\\
	\dot \Psi(R,R_c)= e_R^Te_\Omega,\label{eqn:Psidot},\\
	||C(R_c^TR)||\leq1.\label{eqn:CRCTR}
\end{gather}
Equations \refeqn{eRdot} and~\refeqn{Psidot}--\refeqn{CRCTR} are presented in~\cite{LeeLeoPICDC10}, and \refeqn{eomegadot} is derived from taking derivative of~\refeqn{eomega} and substituting~\refeqn{E2} and~\refeqn{Mc}.

For a non-negative constant $c_2$, the Lyapunov function for the attitude dynamics is defined as
\begin{gather}
	\mathcal{V}_2=\frac{1}{2} e_\Omega^T J e_\Omega+k_R \Psi(R,R_c)+c_2 e_R^T Je_\Omega+\mathcal{V}_{0_2},\label{eqn:V2}
\end{gather}
where $\mathcal{V}_{0_2}$ is given by~\refeqn{Lyap0}, and 
\begin{gather}
	\frac{1}{2}\norm{e_R}^2\leq\Psi(R,R_c)\leq\frac{1}{2-\psi_1}\norm{e_R}^2,
\end{gather}
with $\psi_1=\frac{1}{k_R}[\frac{1}{2}e_\Omega(0)^TJe_\Omega(0)+k_R\Psi(R(0),R_c(0))]$. The bounds of $\mathcal{V}_2$ are
\begin{gather}
	\lambda_m(\mathcal{M}_{21}) ||\mathcal{Z}_{21}||^2+\mathcal{V}_{0_2}\leq\mathcal{V}_2 \leq \lambda_M(\mathcal{M}_{22}) ||\mathcal{Z}_{21}||^2+\mathcal{V}_{0_2}, \label{eqn:V2bound}
\end{gather}
where \begin{gather}
	\mathcal{M}_{21}=\frac{1}{2}\begin{bmatrix} k_R & -c_2 \lambda_{M_J}
		\\ -c_2 \lambda_{M_J} & \lambda_{m_J} \end{bmatrix}, 
	\mathcal{M}_{22}=\frac{1}{2}\begin{bmatrix} \frac{2k_R}{2-\psi_1} & c_2\lambda_{M_J}
		\\ c_2\lambda_{M_J} & \lambda_{M_J}\end{bmatrix},\label{eqn:M2}\\
	\mathcal{Z}_{21}=[||e_R||,||e_\Omega||]^T,\label{eqn:Z21}
\end{gather}
with $\lambda_{m_J}=\lambda_m(J),\lambda_{M_J}=\lambda_M(J)$. Provided that $c_2$ is sufficiently small to satisfy the following inequality, the matrices $\mathcal{M}_{21},\mathcal{M}_{22}$ are positive-definite,
\begin{gather}
	c_2<\min\{\frac{\sqrt{k_R\lambda_{m_J}}}{\lambda_{M_J}},  \sqrt{\frac{2k_R}{\lambda_M(2 - \psi_1)}}\},\label{eqn:c20}
\end{gather}
where $\psi_1 < 2$. 

The time-derivate of the $\Lya$ function is given by
\begin{align}
	\dot{\mathcal{V}}_2=& (e_\Omega+c_2e_R)^T J \dot e_\Omega+k_R \dot \Psi(R,R_c)+c_2 \dot e_R^T Je_\Omega\nonumber\\&+\dot{\mathcal{V}}_{0_2},
\end{align}
where $\dot{\mathcal{V}}_{0_2}$ is given by~\refeqn{Lyap0_dot}. Substituting error dynamics~\refeqn{eRdot}--\refeqn{CRCTR},~\refeqn{E2}, and~\refeqn{Mc}, results in
\begin{align}
	\dot{\mathcal{V}}_2&= (e_\Omega+c_2e_R)^T (-k_Re_R-k_\Omega e_\Omega-\tilde \Delta_2)\nonumber \\
	&\quad +k_R e_R^T e_\Omega+c_2 C(R_c^TR)e_\Omega^T Je_\Omega+\dot{\mathcal{V}}_{0_2}.
\end{align}
From~\refeqn{quad},~\refeqn{CRCTR}, and~$\norm{J}\leq\lambda_{M_J}$,
\begin{align}
	\dot{\mathcal{V}}_2\leq&-c_2k_R e_R^Te_R+c_2 k_\Omega ||e_R||||e_\Omega||-(k_\Omega-c_2\lambda_{M_J})e_\Omega^Te_\Omega\nonumber \\
	&-(e_\Omega+c_2e_R)^T(\tilde \Delta_2)+\dot{\mathcal{V}}_{0_2}.
\end{align}
Since~\refeqn{a1a2}, the last term of this expression is the same as~\refeqn{B}. Substituting its equivalent expression given by~\refeqn{last_term}, results in
\begin{align}
\dot{\mathcal{V}}_2\leq&-c_2k_R e_R^Te_R+c_2 k_\Omega ||e_R||||e_\Omega||-(k_\Omega-c_2\lambda_{M_J})e_\Omega^Te_\Omega\nonumber \\
&-\frac{\kappa_2}{2}||\tilde Z_2||^2+\frac{\kappa_2 Z_{M_2}^2}{2}\nonumber\\&+\norm{a_2}\{C_{1_2}+C_{2_2}||\tilde Z_2|| (1+\norm{E(R)}+\norm{\Omega})\}.\label{eqn:vdot2_mid0}
\end{align}

It is assumed that $\norm{\dot{\bar{\Delta}}_1}\leq\delta_2$ and the desired trajectory is designed such that $\norm{\dddot{x}_d}\leq\delta_3$, where $\delta_2,\,\delta_3>0$. Thus $\norm{m\dddot{x}_d+\dot{\bar{\Delta}}_1}\leq B_2$. Taking derivative of~\refeqn{A}, it can be shown that
\begin{gather}
\norm{\dot{A}}\leq k_x\norm{e_v}+k_v \norm{\dot{e}_v}+B_2.\label{eqn:Adot_bound}
\end{gather}
Since~\refeqn{Rc}, $\dot R_c=[\dot b_{1c},\dot b_{2c},\dot b_{3c}]$, where
\begin{gather}\dot b_{1c}=\dot b_{2c}\times b_{3c}+b_{2c}\times \dot b_{3c},\label{eqn:b1cdot}\\
\dot b_{2c}=-\frac{\dot C}{||C||}+\frac{C(C.\dot C)}{||C||^3},\label{eqn:b2cdot}\\
\dot b_{3c}=-\frac{\dot A}{||A||}+\frac{A(A.\dot A)}{||A||^3}.\label{eqn:b3cdot}
\end{gather}
Since~\refeqn{acceleration_bound},~\refeqn{Adot_bound}, and~\refeqn{b3cdot}
\begin{gather}
\norm{\dot b_{3c}}\leq 2 \frac{k_x\norm{e_v}+k_v \norm{\dot{e}_v}+B_2}{k_x\norm{e_x}+k_v \norm{e_v}+B_1}\equiv B_3.\label{eqn:ndotb_3c}
\end{gather}
It is assumed that the desired trajectory is designed such that $\norm{\dot{b}_{1_d}}\leq \delta_4$, where $\delta_4>0$. Taking derivative of~\refeqn{C} and using~\refeqn{ndotb_3c}, it can be shown that $\norm{\dot C}\leq B_3+\delta_4$. Since~\refeqn{C},~\refeqn{b3cdot}, and $\norm{C}\leq 1$,
\begin{gather}
\norm{\dot b_{2c}}\leq 2(B_3+\delta_4),\label{eqn:ndotb_2c}
\end{gather}
From~\refeqn{b1cdot},~\refeqn{ndotb_3c}--\refeqn{ndotb_2c}
\begin{gather}
\norm{\dot b_{1c}}\leq 3B_3+2\delta_4.\label{eqn:ndotb_1c}
\end{gather}
Thus, from~\refeqn{ndotb_3c}--\refeqn{ndotb_1c}, 
it can be shown that $\norm{\dot{R}_c}\leq B_4$, where $B_4$ is positive. Since~\refeqn{EC3}, $\norm{\Omega_c}\leq B_4$. Since~\refeqn{eomega}, $\norm{\Omega}\leq\norm{e_\Omega}+B_4$.

$\norm{E(R)}\leq E_{max}$, where $ E_{max}$ is positive, and substituting these in~\refeqn{vdot2_mid0}, expanding $a_2$, using~\refeqn{x_y} result in
\begin{align}
	\dot{\mathcal{V}}_2\leq&-c_2k_R e_R^Te_R+c_2 k_\Omega ||e_R||||e_\Omega||-(k_\Omega-c_2\lambda_{M_J})e_\Omega^Te_\Omega\nonumber \\
	&-\frac{\kappa_2}{2}||\tilde Z_2||^2+\frac{\kappa_2 Z_{M_2}^2}{2}\nonumber\\&+C_{1_2}\norm{e_\Omega} +C_{3_2} \norm{e_\Omega}^2+C_{4_2}\norm{e_\Omega}\norm{\tilde Z_2}\nonumber\\
	&+c_2(C_{1_2}\norm{e_R}+C_{4_2}\norm{e_R}\norm{\tilde Z_2})\nonumber\\&+c_2C_{3_2}\norm{e_R}\norm{e_\Omega},\label{eqn:vdot2_mid1}
\end{align}
where $C_{1_2}\geq 2W_{M_2}+\epsilon_2,\,C_{2_2}\geq 0.25 (V_{M_2}+W_{M_2}),\,C_{3_2}\geq C_{2_2} Z_{M_2},\,C_{4_2}\geq C_{2_2}(1+E_{max}+B_4)$.
Using~\refeqn{ax2+bx}, the following expressions are rearranged as
\begin{align}
	-&k_{R} e_R^Te_R+C_{1_2}||e_R||\leq-\frac{k_{R}}{2}e_R^Te_R+\frac{C_{1_2}^2}{2k_{R}},\label{eqn:two1}\\
	-&k_{\Omega_\beta}e_\Omega^Te_\Omega+C_{1_2}||e_\Omega||\leq-\frac{k_{\Omega_\beta}}{2}e_\Omega^Te_\Omega+\frac{C_{1_2}^2}{2k_{\Omega_\beta}},\label{eqn:two2}
\end{align}
where $k_{\Omega_\beta}=k_\Omega-c_2\lambda_{M_J}-C_{3_2}$.
Then substituting~\refeqn{two1}--\refeqn{two2} in~\refeqn{vdot2_mid1}
\begin{align}
\dot{\mathcal{V}}_2\leq&-\frac{c_2k_{R}}{2}e_R^Te_R-\frac{k_{\Omega_\beta}}{2}e_\Omega^Te_\Omega-\frac{\kappa_2}{2}||\tilde Z_2||^2\nonumber \\
&+k_{R\Omega} ||e_R||||e_\Omega||\nonumber\\
&+C_{4_2}||e_\Omega||||\tilde Z_2||+c_2C_{4_2}||e_R||||\tilde Z_2||+C_{5_2}.\label{eqn:V2dot0}
\end{align}
where $k_{R\Omega}=c_2(\kappa_\Omega+C_{3_2}),\, C_{5_2}=\frac{c_2C_{2_1}^2}{2k_{R}}+\frac{C_{2_1}^2}{2k_{\Omega_\beta}}+\frac{\kappa_2 Z_{M_2}^2}{2}$. 


\subsection{Stability Proof for Quadrotor Dynamics} \label{Apen:quadrotor}
Here, we combine the position error dynamics and the attitude error dynamics to show the stability properties of the controlled quadrotor. The Lyapunov function is chosen as $\mathcal{V}=\mathcal{V}_1+\mathcal{V}_2$, where $\mathcal{V}_1,\mathcal{V}_2$ are given by~\refeqn{V1},~\refeqn{V2}. From~\refeqn{V1bound} and~\refeqn{V2bound}, the bound on $\mathcal{V}$ is given by
\begin{align}
 \lambda_m&(\mathcal{M}_{11}) ||\mathcal{Z}_{11}||^2+\lambda_m(\mathcal{M}_{21}) ||\mathcal{Z}_{21}||^2+\mathcal{V}_{0_1}+\mathcal{V}_{0_2}\leq \mathcal{V}\nonumber\\&\leq \lambda_M(\mathcal{M}_{12}) ||\mathcal{Z}_{11}||^2+\lambda_M(\mathcal{M}_{22}) ||\mathcal{Z}_{21}||^2+\mathcal{V}_{0_1}+\mathcal{V}_{0_2}.
\end{align}
The upper bound can be rewritten as 
\begin{align}
\mathcal{V}\leq
&\frac{1}{2} \mathcal{Z}_{1}^T \mathcal{N}_{1}^\prime \mathcal{Z}_{1}+\frac{1}{2} \mathcal{Z}_{2}^T \mathcal{N}_{2}^\prime \mathcal{Z}_{2}+\frac{1}{2}\mathcal{Z}_{3}^T\mathcal{N}_{3}^\prime\mathcal{Z}_{3},
\end{align}
where
\begin{gather*}
\mathcal{N}_{1}^\prime=\begin{bmatrix}\frac{c_2 k_{R}}{2}&mc_1&0\\mc_1&\frac{m}{2}&0\\0&0&\frac{1}{\min\{\gamma_{w_1},\gamma_{v_1}\}}\end{bmatrix},\nonumber\\
\mathcal{N}_{2}^\prime=\begin{bmatrix}\frac{1}{2-\psi_2}&c_2\lambda_{M_J}&0\\c_2\lambda_{M_J}&\lambda_{M_J}&0\\0&0&\frac{1}{\min\{\gamma_{w_2},\gamma_{v_2}\}}\end{bmatrix},\nonumber\\
\mathcal{N}_{3}^\prime=\begin{bmatrix}\frac{k_x}{2}&0&0\\0&\frac{m}{2}&0\\0&0&\frac{1}{2-\psi_2}\end{bmatrix},\label{eqn:N1prime}\\
\mathcal{Z}_{1}=[||e_x||,\norm{e_v},||\tilde Z_1||]^T,\quad  \mathcal{Z}_{2}=[\norm{e_R},||e_\Omega||,||\tilde Z_2||]^T,\\
\mathcal{Z}_{3}=[\norm{e_x},||e_v||,||e_R||]^T.
\end{gather*}

The derivative of the $\Lya$ function is $\dot{\mathcal{V}}=\dot{\mathcal{V}}_1+\dot{\mathcal{V}}_2$. 
From~\refeqn{V1dot} and~\refeqn{V2dot0}, it can be written as
\begin{align}
\dot{\mathcal{V}}\leq
&-\frac{1}{2} \mathcal{Z}_{1}^T \mathcal{N}_{1} \mathcal{Z}_{1}-\frac{1}{2} \mathcal{Z}_{2}^T \mathcal{N}_{2} \mathcal{Z}_{2}-\frac{1}{2}\mathcal{Z}_{3}^T\mathcal{N}_{3}\mathcal{Z}_{3} +C_{5},
\end{align}
where $C_5=C_{5_1}+C_{5_2}$, and
\begin{gather}
\mathcal{N}_{1}=\begin{bmatrix}\frac{c_1 k_{x_\beta}}{2}&-\frac{k_{xv}}{2}&-c_1C_{4_1}\\-\frac{k_{xv}}{2}&\frac{k_{v_\beta}}{2}&-C_{4_1}\\-c_1C_{4_1}&-C_{4_1}&\kappa_1\end{bmatrix},\label{eqn:N1}\\
\mathcal{N}_{2}=\begin{bmatrix}\frac{c_2 k_{R}}{2}&-k_{R\Omega}&-c_2C_{4_2}\\-k_{R\Omega}&k_{\Omega_\beta}&-C_{4_2}\\-c_2C_{4_2}&-C_{4_2}&\kappa_2\end{bmatrix},\label{eqn:N2}\\
\mathcal{N}_{3}=\begin{bmatrix}\frac{c_1 k_{x_\beta}}{2}&-\frac{k_{xv}}{2}&-c_1B_1\\-\frac{k_{xv}}{2}&\frac{c_1 k_{v_\beta}}{2}&-B_1-k_xe_{x_{max}}\\-c_1B_1&-B_1-k_xe_{x_{max}}&\frac{c_2 k_{R}}{2}\end{bmatrix}.\label{eqn:N3}
\end{gather}
If $c_1$ and $c_2$ are chosen such that $\mathcal{N}_{1} ,\,\mathcal{N}_{2} ,\,\mathcal{N}_{3}$ become positive definite then the right hand side of the inequality reduces to
\begin{align}
\dot{\mathcal{V}}\leq-\nu \mathcal{V}+C_{5},\label{eqn:Vdot_final}
\end{align}
where $\nu=\min\{\frac{\lambda_{m}(\mathcal{N}_{1})}{\lambda_{M}(\mathcal{N}_{1}^\prime)},\, \frac{\lambda_{m}(\mathcal{N}_{2})}{\lambda_{M}(\mathcal{N}_{2}^\prime)},\,\frac{\lambda_{m}(\mathcal{N}_{3})}{\lambda_{M}(\mathcal{N}_{3}^\prime)}\}$. If $\mathcal{V}>\frac{C_{5}}{\nu}$, then $\dot{\mathcal{V}}<0$. Therefore, according to~\cite{khalil1996noninear}, $e_x,e_v,e_R,e_\Omega,\tilde Z_1$ and $\tilde Z_2$ are bounded and converge exponentially to the set $\mathcal{D}$
\begin{align}
\mathcal{D}&=\{e_x,e_v,e_R,e_\Omega \in \Ree^3,\tilde Z_1 \in \Ree^{N_{1_1}+N_{2_1}+2\times N_{2_1}+N_{3_1}},\nonumber\\& \tilde Z_2 \in \Ree^{N_{1_2}+N_{2_2}+2\times N_{2_2}+N_{3_2}}\arrowvert \norm{e_x}^2+\norm{e_v}^2+\norm{e_R}^2+\norm{e_\Omega}^2\nonumber\\
&+\frac{1}{\gamma_1}\|\tilde Z_1\|^2+\frac{1}{\gamma_2}\|\tilde Z_2\|^2\leq \frac{C_{5}}{\nu}\},\label{eqn:D}
\end{align}
where $\gamma_1=\max\{\gamma_{v_1},\gamma_{w_1}\},\,\gamma_2=\max\{\gamma_{v_2},\gamma_{w_2}\}$. By adjusting controller gains, the set can be made arbitrarily small.

\bibliography{BibMaster_new,Mahdis}
\bibliographystyle{IEEEtran}

\end{document}